\documentclass{amsart}
\usepackage{url}
\usepackage{amsopn}
\usepackage{amsmath,upref}
\usepackage{mathrsfs}
\usepackage{amssymb}
\usepackage[all,2cell,emtex,cmtip]{xy}
\input xypic

\newtheorem{thm}{Theorem}[section]
\newtheorem{prop}[thm]{Proposition}
\newtheorem{lemma}[thm]{Lemma}
\newtheorem{cor}[thm]{Corollary}

\theoremstyle{remark} 
\newtheorem{rem}[thm]{Remark}

\newtheorem{paragr}[thm]{}

\theoremstyle{definition} 

\theoremstyle{plain}

\newcommand{\derR}{\mathrm{R}}

\newcommand{\To}{\rightarrow}

\newcommand{\Hom}{\operatorname{\mathrm{Hom}}}

\newcommand{\CC}{\mathbf{C}}

\newcommand{\derL}{\mathrm{L}}

\renewcommand{\AA}{\mathbf{A}}

\newcommand{\ZZ}{\mathbf{Z}}
\newcommand{\QQ}{\mathbf{Q}}

\newcommand{\et}{\mathrm{\acute{e}t}}
\newcommand{\rig}{\mathrm{rig}}

\newcommand{\coh}{\mathrm{H}}
\newcommand{\spec}[1]{\mathrm{Spec}(#1)}

\newcommand{\DM}{\mathrm{DM}}
\newcommand{\DA}{\mathrm{DA}}
\newcommand{\Der}{\mathrm{D}}
\newcommand{\h}{\mathit{h}}

\newcommand{\gm}{\mathit{gm}}
\newcommand{\Cat}[1]{\mathcal{#1}}

\def\TO#1{\mathrel{\hbox to #1mm{\rightarrowfill}}}
\def\OT#1{\mathrel{\hbox to #1mm{\leftarrowfill}}}

\numberwithin{equation}{thm}

\title{Completion of motivic sheaves}
\author[D.-C. Cisinski]{Denis-Charles Cisinski}
\address{Fakult{\"a}t f{\"u}r Mathematik,
	Universit{\"a}t Regensburg,
	93040 Regensburg,
	Germany}
\email{denis-charles.cisinski@ur.de}
\urladdr{https://cisinski.app.uni-regensburg.de/}
\thanks{Research supported in part by the SFB 1085 ``Higher Invariants'',
funded by the Deutsche Forschungsgemeinschaft (DFG)}

\begin{document}
\begin{abstract}
We study the process of $\ell$-adic completion of motivic sheaves.
We observe that, in equal characteristic,
when restricted to constructible objets,
it is compatible with the six operations.
This implies that one can reconstruct $\ell$-adic sheaves of geometric origin
over a scheme of finite type over a field from $\ell$-adic cohomology
of smooth schemes. In the case of finite fields, this includes perverse $\ell$-adic
sheaves of geometric orgin. However, the analogous behaviour fails systematically
in mixed characteristic: the reason is that it would imply strong independence of
$\ell$ results that can be proven to be too optimistic.
\end{abstract}
%
%
\maketitle
\section*{Introduction}

It known that there is a strong analogy between motives and abelian
groups: abelian groups are motives in complex analytic geometry.
It is thus tempting to try to draw a parallel between
derived categories of motives and the derived category of abelian groups.
An important feature of motivic sheaves, at least if we work locally for the
\'etale topology, is that (derived) $\ell$-adic completion yields
$\ell$-adic realization functors (for $\ell$ a prime number distinct from
the residue characteristics of the base). Therefore, completions of motivic sheaves
play a central role in the motivic story, and it is natural to understand
the general structural properties that one can observe through completions
in this setting.
In the context of abelian groups, derived $\ell$-adic completion has a nice
behavior on perfect complexes: given a perfect complex of $\ZZ$-modules $E$,
since $E$ is dualizable, the derived tensor product by $E$ is a right adjoint
and we see that the $\ell$-adic completion of $E$ is isomorphism to the tensor
product of $E$ with the $\ell$-adic completion of $\ZZ$:
\[
\ZZ_\ell\otimes^\derL E\cong E_\ell
\]
Given a scheme $X$ and a motivic sheaf $M$, we investigate whether
the canonical comparison map
\[
\ZZ_\ell\otimes^\derL M\to M_\ell
\]
is invertible in $\DM(X)$. In the first section, we prove that for $X$
of finite type over a field and $M$ constructible, this is indeed the case
(Proposition~\ref{prop:petit everest}).
The proof uses de Jong's alteration theorem and consists in proving
another structural property of $\ell$-adic completion: given
any morphism $f\colon X\to Y$ of schemes of finite type over a field,
the pullback functor
\[
f^*\colon\DM(Y)\to\DM(X)
\]
commutes with derived $\ell$-adic completion of constructible objects.
A formal consequence of this story is that, given $\ell$-adic cohomology
seen as a cohomology theory of smooth schemes over a field $k$, one
can reconstruct the six-functor formalism of $\ell$-adic sheaves of geometric
origin for schemes of finite type over $k$ (Theorem~\ref{thm:everest categorique}). 
Furthermore, at least if $k$ is the algebraic closure of a finite field,
using Deligne's theory of weights,
this reconstruction theorem is compatible with the perverse $t$-structure
(Corollary~\ref{cor:geometric perverse}), and one can develop of Weil-\'etale
version of this story (Remark~\ref{rem:Weil etale}).

It is thus natural to ask whether these nice structural properties can be
proved in mixed characteristic. And this may be the most
important contribution of this article; the answer is:
not at all, neither of them (Theorem~\ref{thm:non-everest}). The proof is not constructive:
we prove that it would imply the existence of a comparison isomorphism
between $\ell$-adic cohomology and rigid cohomology (for schemes over a field
of characteristic $p\neq\ell$), with structural properties that would be too good
to be true (the argument is somehow a variation on the non-existence of
a Weil cohomology with $\QQ$-coefficients over a big enough
field of positive characteristic, using products of supersingular elliptic curves).\\

At the end of the day, this means that we are facing the following situation.
Consider a noetherian scheme $S$ of dimension $1$ and a closed point $s$ in $S$ in the regular locus of $S$. We assume that
the residue field $k$ of $s$ is
algebraically closed and of positive 
characteristic, and that the local ring $\mathcal{O}_{S,s}$ has
at most the cardinality of the continuum. If $\ell$ is any prime number
distinct from the charactersitc of $k$ and $i\colon s\to S$ the inclusion map,
there is a motivic sheaf $i^*(\ZZ_{\ell,S})$ over $s$
equipped with a map to ordinary $\ell$-adic cohomology of smooth $k$-schemes
$i^*(\ZZ_{\ell,S})\to\ZZ_{\ell,s}$
that we can embed in a fiber sequence:
\[
E\to i^*(\ZZ_{\ell,S})\to\ZZ_{\ell,s}
\]
We can then define, for any $k$-scheme
of finite type $X$, motivic cohomology of $X$ with coefficients in $E$
through 
\[
\coh^j(X,E(j))=\coh^j(\derR\Hom_{\DM(X)}(\QQ_X,E_X(j)))
\]
where $E_X$ denotes the pullback of $E$ from $s=\spec k$ to $X$.
Note that, since the map $i^*(\ZZ_{\ell,S})\to\ZZ_{\ell,s}$
induces an isomorphism after $\ell$-adic completion, $E$ has to be
$\QQ$-linear. If ever $S$ is of finite type over $k$, then we get
$\coh^j(X,E(j))\cong 0$ for all $X$ and any indices $i,j$
(this is what Proposition~\ref{prop:petit everest} means).
But if $S$ is generically of characteristic zero, then this cohomology
theory is not trivial, by Theorem~\ref{thm:non-everest}.
Constructing explicit non-trivial elements in cohomology groups
of the form $\coh^j(X,E(j))$ is a challenging question:
this is purely of arithmetic nature, apparently violating the analogy between
curves and rings of integers.\\

\noindent\emph{Acknowledgements.}
This work started a long time ago, while I undertook my motivic journey
with Fr\'ed\'eric D\'eglise.
The last paragraph is based on discussions with
Daniel Litt and Alexander Petrov.

\section{Motivic pullbacks of $\ell$-adic spectra}
\begin{paragr}
Given a scheme $X$ and a commutative ring $\Lambda$,
there is the stable presentable $\infty$-category of $\Lambda$-linear
\'etale motives $\DM(X,\Lambda)$. It can be constructed as the $\ZZ(1)$-stabilization
of a $\AA^1$-homotopy invariant \'etale hypersheaves of complexes of $\Lambda$-modules
on the smooth site of $X$, as in \cite{AyoubEtale},
or in terms of hypersheaves for the
$\h$-topology, as in \cite{CD4}.
They are equivalent since both
are known to agree with rational coefficients
(by \cite[theorems 16.1.2 and 16.2.18]{CD3})
as well as with torsion coefficients because both satisfy a suitable
version of the rigidity theorem; see
\cite[Cor. 5.5.4]{CD4} and
\cite[theorem 3.1]{Bachmann}. Recalling this equivalence is not innocent:
we will use freely, without warning, that any cohomology theory
that is representable in \'etale motivic sheaves through the six operations
automatically satisfies $\h$-hyperdescent.
Another important aspect of the description of motivic sheaves
through $\h$-hypersheaves is the following property.
\begin{lemma}\label{lemma:left adjoint h-locally}
Let $k$ be any field. Then, for any $k$-scheme of finite type $X$,
with structural map
$f\colon X\to\spec k$, the pullback
functor
\[
f^*\colon\DM(\spec k,\ZZ)\to\DM(X,\ZZ)
\]
has a left adjoint. In particular, it preserves limits.
\end{lemma}
\begin{proof}
By topological invariance~\cite[Prop.~6.3.16]{CD4}
we may assume that $k$ is perfect.
Given any noetherian $k$-scheme $X$,
let $\underline{\DM}(X)$ be the
$\ZZ(1)$-stabilization
of $\AA^1$-homotopy invariant $\h$-hypersheaves of complexes of
abelian groups on the site of separated schemes of finite type over $X$.
We may identify $\DM(X,\ZZ)$ as the smallest full stable $\infty$-category
of $\underline{\DM}(X)$
closed under small colimits and Tate twists
containing motives of separated smooth schemes of finite over $X$;
see \cite[Def.~5.1.3]{CD4}.
This identification is compatible with the formation of pullback
functors. On the other hand, since $k$ is perfect,
de Jong's alteration theorem~\cite{dJ}
implies
\[
\DM(\spec k,\ZZ)\cong\underline{\DM}(\spec k)\, .
\]
It is an easy exercise to see that the left adjoint of
\[
f^*\colon\DM(\spec k,\ZZ)\cong\underline{\DM}(\spec k)\to\underline{\DM}(X)
\]
restricts on $\DM(X,\ZZ)$ to a left adjoint of the pullback functor $f^*$
in $\DM$. 
\end{proof}

Let $\ell$ be a prime number distinct from the residue
characteristics of $X$.
and $\ZZ_{(\ell)}$ the localization 
of the ring of integers $\ZZ$ at the prime ideal $(\ell)$.
The $\infty$-category $\Der(X,\ZZ_\ell)$ of $\ell$-adic sheaves on $X$
is defined as the quotient of $\DM(X,\ZZ_{(\ell)})$ by the
subcategory of those motivic sheaves $M$ such that
$M\otimes^\derL\ZZ/\ell\ZZ\cong 0$.
By virtue of \cite[Prop.~7.2.21]{CD4},
the triangulated category associated to the stable
category $\Der(X,\ZZ_\ell)$ defined above
is canonically equivalent to
Ekedahl's triangulated category of $\ell$-adic sheaves~\cite{Eke}.

The localization functor
\begin{equation}
\DM(X,\ZZ)\To \Der(X,\ZZ_\ell)
\end{equation}
has a fully faithful right adjoint which identifies $\Der(X,\ZZ_\ell)$ with the full subcategory of \emph{$\ell$-complete motivic sheaves}: those $M$ in $\DM(X,\ZZ)$ such that the canonical map
\begin{equation}
M\To M_\ell=\derR\varprojlim_\nu M\otimes^\derL\ZZ/\ell^\nu\ZZ
\end{equation}
is an isomorphism. We define
\begin{equation}
\ZZ_{\ell,X}=(\ZZ_X)_\ell
\end{equation}
as the $\ell$-adic completion of the constant motivic sheaf $\ZZ_X$ on $X$.
by virtue of the rigidity theorem~\cite[Prop.~7.2.21]{CD4},
this is an $E_\infty$-ring spectrum which represents
continuous $\ell$-adic cohomology of smooth $X$-schemes.

For any motivic sheaf $M$ in $\DM(X,\ZZ)$, there is a canonical
comparison map
\begin{equation}\label{equa:l-completion}
\ZZ_{\ell,X}\otimes^\derL M\To M_\ell\, .
\end{equation}
\end{paragr}
\begin{prop}\label{prop:petit everest}
Let $k$ be a field
and $\ell$ a prime number.
The following properties hold.
\begin{itemize}
\item[(i)] For any $k$-scheme of finite type $X$ any any
constructible motivic sheaf $M$ on $X$, there is a canonical
isomorphism $\ZZ_{\ell,X}\otimes^\derL M\cong M_\ell$.
\item[(ii)] For any morphism $f:X\to Y$ of $k$-schemes of finite type,
the pull-back functor $f^*:\DM(Y,\ZZ)\to\DM(X,\ZZ)$
induces an isomorphism
$$f^*(\ZZ_{\ell,Y})\cong\ZZ_{\ell,X}\, .$$
\end{itemize}
\end{prop}
\begin{proof}
To prove (i), we observe that the class of those $M$ such that the canonical
map \eqref{equa:l-completion} is invertible is closed under finite (co)limits
and direct factors. Since constructible motivic sheaves are generated
by those of the form $f_*(\ZZ_U)(n)$ for a proper map $f:U\to X$
and an integer $n$, and since all functors of interest commute with Tate twists
(see e.g. \cite[Prop.~4.2.13]{CD3}),
it suffices to prove the particular case
where $M=f_*(\ZZ_U)$ for a proper map $f:U\to X$.
We have then the projection formula:
\begin{align*}
\ZZ_{\ell,X}\otimes^\derL M&\cong\ZZ_{\ell,X}\otimes^\derL f_*(\ZZ_U)\\
&\cong f_*f^*(\ZZ_{\ell,X})\, .
\end{align*}
Therefore, since $f_*$ commutes with limits,
it is sufficient to prove property (ii) in the case where $f$ is proper
to prove property (i).

We will now prove property (ii).
Let $p:X\To\spec k$ and $q:Y\To\spec k$ be the structural maps.
If we know property (ii) for $p$ and $q$, namely,
$$p^*(\ZZ_{\ell,k})\cong\ZZ_{\ell,X}
\qquad
q^*(\ZZ_{\ell,k})\cong\ZZ_{\ell,Y}\, ,$$
we deduce
$$f^*(\ZZ_{\ell,Y})\cong f^*q^*(\ZZ_{\ell,k})\cong
p^*(\ZZ_{\ell,k})\cong\ZZ_{\ell,X}\, .$$
Therefore, it is sufficient to prove the case where $Y=\spec k$,
which follows immediately from
Lemma~\ref{lemma:left adjoint h-locally}.
\end{proof}
\begin{rem}
The proof of the previous proposition provides a more general formula:
it says that, given any small category $J$ and any $J$-indexed functor
\[
j\mapsto A_j
\]
with values in the $\infty$-category $\DM(k,\ZZ)$, for any morphism of finite type
\[
f\colon X\to\spec k
\]
and any constructible motivic sheaf $M$ on $X$, the canonical map
\[
M\otimes^\derL f^*\Big(\derR\varprojlim_{j\in J}A_j\Big)
\to
\derR\varprojlim_{j\in J}M\otimes^\derL f^*(A_j)
\]
is an equivalence. This is non-trivial since neither the functor $f^*$
or $M\otimes^\derL(-)$ have a right adjoint.
\end{rem}
\begin{paragr}
The meaning of the previous proposition is that one can describe
$\ell$-adic sheaves of geometric origin from motivic sheaves
and $\ell$-adic cohomology of smooth schemes over $k$ as follows.
For a scheme of finite type $X$ over $k$, we define the
following stable
$\infty$-category:
\[
\DM(X,\ZZ_{\ell,X})=\{\text{$\ZZ_{\ell,X}$-modules in $\DM(X,\ZZ)$}\}\, .
\]
Since $\ZZ_X$ and $\ZZ_{\ell,X}$ agree modulo $\ell$ by definition,
we see that the full subcategory of $\ell$-complete objects
in $\DM(X,\ZZ_{\ell,X})$ is canonically equivalent to the
full subcategory of $\ell$-complete objects in $\DM(X,\ZZ)$, namely
$\Der(X,\ZZ_\ell)$.
We denote by $\Der^b_\gm(X,\ZZ_\ell)$ the full
subcategory of $\DM(X,\ZZ_{\ell,X})$ spanned by
constructible objects in the motivic sense (i.e. the smallest thick stable
subcategory containing objects of the form
$\ZZ_{\ell,X}\otimes M$ for a constructible motivic sheaf $M$
in $\DM(X,\ZZ)$).
\end{paragr}
\begin{prop}\label{prop:ladic fully faithful}
For any scheme of finite type $X$ over a field $k$, 
for any prime $\ell\neq\mathit{char}(k)$, the
$\ell$-adic completion functor induces a fully faithful exact
functor
\[
\Der^b_\gm(X,\ZZ_\ell)\to\Der^b_c(X,\ZZ_\ell)\, .
\]
The essential image consists is the smallest
thick stable subcategory of the ordinary stable $\infty$-category of
constructible $\ell$-adic sheaves containing objects of the form
$f_!f^!(\ZZ_\ell)(n)$ with $f\colon Y\to X$ smooth and $n$ any integer.
Furthermore, the subcategory $\Der^b_\gm(X,\ZZ_\ell)$ is closed under
the six operations in $\Der^b_c(X,\ZZ_\ell)$.
\end{prop}
\begin{proof}
For the fully faithfulness, it suffices to check it on generators:
for any constructible motivic sheaves $M$ and $N$ in $\DM(X,\ZZ)$,
we have to check that the map induced by $\ell$-adic completion
\[
\derR\Hom_{\Der^b_\gm(X,\ZZ_\ell)}(\ZZ_{\ell,X}\otimes^\derL M,\ZZ_{\ell,X}\otimes^\derL N)
\to \derR\Hom_{\Der^b_c(X,\ZZ_\ell)}(M_\ell,N_\ell)
\]
is an equivalence.
We have canonical identifications
\[
\derR\Hom_{\DM(X,\ZZ)}(M,\ZZ_{\ell,X}\otimes^\derL N)
\cong
\derR\Hom_{\Der^b_\gm(X,\ZZ_\ell)}(\ZZ_{\ell,X}\otimes^\derL M,\ZZ_{\ell,X}\otimes^\derL N)
\]
and
\begin{align*}
\derR\Hom_{\Der^b_c(X,\ZZ_\ell)}(M_\ell,N_\ell)
&\cong
\derR\varprojlim_\nu\derR\Hom_{\DM(X,\ZZ)}(M,\ZZ/\ell^\nu\otimes^\derL N)\\
&\cong\derR\Hom_{\DM(X,\ZZ)}(M,N_\ell) \, .
\end{align*}
In other words, we must check that the canonical map
\[
\ZZ_{\ell,X}\otimes^\derL N\to N_\ell
\]
is an equivalence whenever $N$ is constructible.
Therefore, the fully faithfulness property finally amounts to assertion (i) in the
previous proposition.
The stability property under the six operations essentially is a reformulation
of \cite[Thm.~7.2.16]{CD4}.
\end{proof}

\section{Perverse sheaves of geometric origin}
\begin{paragr}
Let $k$ be an algebraically closed field of
positive characteristic
and $\ell\neq \mathit{char}(k)$ a prime number.
Given any separated $k$-scheme of finite type $X$, we consider
\[
\overline\QQ_{\ell,X}=\QQ_{\ell,X}\otimes_{\QQ_\ell}\overline\QQ_\ell
\]
where $\overline\QQ_\ell$ is an algebraic closure of $\QQ_\ell$,
and we define
\[
\Der^b_\gm(X,\overline\QQ_\ell)
\]
as the $\infty$-category of compact objects in the
$\infty$-category of $\overline\QQ_{\ell,X}$-modules in
the category of motivic sheaves
$\DM(X,\ZZ)$.\\

Our aim is now to understand the structure of $\Der^b_\gm(X,\overline\QQ_\ell)$.
It is closed under the six operations and the $\ell$-adic realization functor
\begin{equation}\label{eq:comp etale vs rigid catbis}
\DM_{c}(X,\QQ)\to\Der^b_\gm(X,\overline\QQ_\ell)\ ,
\quad M\mapsto\overline{\QQ}_{\ell,X}\otimes M
\end{equation}
commutes with them.

Furthermore, $\Der^b_\gm(s,\overline\QQ_\ell)\cong\mathrm{Perf}(\overline\QQ_\ell)$ is the bounded derived $\infty$-category of finite dimensional $\overline\QQ_\ell$-vector
spaces; see \cite[Thm. 17.2.9 and 17.2.17]{CD3}. Concretely, this also means that
$\Der^b_\gm(X,\overline\QQ_\ell)$ is generated, as an idempotent complete stable
$\infty$-category, by objects of the form $f_*(\overline\QQ_{\ell,Y})(n)$
for $f:Y\to X$ projective and with $Y$ connected and smooth over $k$, $n\in\ZZ$;
see \cite[Cor. 4.4.3]{CD3}.

Let $\Der^b_c(X,\overline\QQ_\ell)$ be the usual
bounded stable $\infty$-category of constructible $\overline\QQ_\ell$-sheaves on $X$.
We can construct it in the following way: for each finite flat extension $R$ of $\ZZ_\ell$,
with uniformizer $l$, there is the stable $\infty$-category  $\Der^b_c(X,R_l)$
(obtained from \cite[7.2.20]{CD4} or, equivalently, from
\cite[Def. 6.3.1]{BS}) and we
form the filtered colimit of symmetric monoidal stable $\infty$-categories
\begin{equation}
\Der^b_c(X,\overline\ZZ_\ell)=\varinjlim_{\ZZ_\ell\subset R}\Der^b_c(X,R_l)\, .
\end{equation}
Finally, we can simply put
\begin{equation}\label{eq:construction Qell faisceaux}
\Der^b_c(X,\overline\QQ_\ell)=\Der^b_c(X,\overline\ZZ_\ell)\otimes\QQ\, .
\end{equation}
By virtue of \cite[Thm. 7.2.24]{CD4}, we have a realization functor
\begin{equation}\label{eq:comp etale vs rigid real}
\DM_{c}(X,\QQ)\to\Der^b_c(X,\overline\QQ_\ell)
\end{equation}
that commutes with the six operations. So far,
Constructions \eqref{eq:construction Qell faisceaux}
and \eqref{eq:comp etale vs rigid real}
are available for any noetherian scheme $X$
over $\ZZ[1/\ell]$ (and we should restrict to
quasi-excellent ones if we want the six operations to restrict to
constructible objects). Coming back to separated schemes of finite
type over $k$,
we see that the realization functor
\[
\DM(X,R)\to\Der(X,R_l)=\DM(X,\ZZ)\otimes_\ZZ\Der(R)^\wedge_{(l)}\, ,
\]
being a symmetric monoidal left adjoint, factors uniquely through
the $\infty$-category of $R_{l,X}$-modules (because $R_{l,X}$
is the image of the unit of $\Der(X,R_l)$ by the right
adjoint).
Restricting
to constructible objects and taking the colimit of the realizations
indexed by finite extensions $R$ of $\ZZ_\ell$,
we see that the functor \eqref{eq:comp etale vs rigid real} factors
uniquely through a functor
\begin{equation}\label{eq:comp etale vs rigid real2}
\Der^b_\gm(X,\overline\QQ_\ell)\to\Der^b_c(X,\overline\QQ_\ell)\, .
\end{equation}
that commutes with the six operations.\footnote{At some point,
the reader who is tired to always do the same
thing might want to rely on more systematic methods
to construct realization functors,
using the universal property of motivic sheaves proved by Drew and Gallauer
\cite{DrGa}, possibly considering ind-objects on $\Der^b_c(X,\overline\QQ_\ell)$.}
\end{paragr}
\begin{thm}\label{thm:everest categorique}
The functor $\Der^b_\gm(X,\overline\QQ_\ell)\to\Der^b_c(X,\overline\QQ_\ell)$
is fully faithful. It identifes the stable $\infty$-category
of compact modules over $\overline{\QQ}_{\ell,X}$ in $\DM(X)$ as the smallest
 thick stable subcategory of
$\Der^b_c(X,\overline\QQ_\ell)$ containing 
objects of the form $f_*(\overline\QQ_{\ell,Y})(n)$
for $f:Y\to X$ projective and with $Y$ connected and smooth over $k$, $n\in\ZZ$.
\end{thm}
\begin{proof}
We observe that Proposition \ref{prop:ladic fully faithful}
implies that \eqref{eq:comp etale vs rigid real2}
is fully faithful (we can also
reproduce the proof of \cite[Thm. 17.1.5]{CD3}).
The second asserion follows from the fact that the
functor \eqref{eq:comp etale vs rigid real2}
is compatible with the six operations and from the similar
description of $\Der^b_\gm(X,\overline\QQ_\ell)$
as the category of compact objects in $\QQ_{\ell,X}$-modules;
see \cite[Cor.~4.4.3]{CD3}.
\end{proof}

\begin{paragr}
From now on, we assume that $k$ is the algebraic closure of a finite field.
We consider the perverse $t$-structure (for the middle perversity) on
$\Der^b_c(X,\overline\QQ_\ell)$; see \cite{BBDG}.
Assume that $X$ is actually defined on some finite subfield $k_\circ\subset k$:
there is $X_\circ$ separated of finite type over $k_\circ$ and $X$ is obtained
from $X_\circ$ by pullback. There is then the $\infty$-category
$\Der^b_c(X_\circ,\overline\QQ_\ell)$
of constructible $\overline\QQ_\ell$-sheaves on $X_\circ$,
and its full subcategory $\Der^b_m(X_\circ,\overline\QQ_\ell)$ of \emph{mixed sheaves}.
The realization functor
\begin{equation}\label{eq:comp etale vs rigid real3}
\DM(X_\circ,\QQ)\to\Der^b_c(X_\circ,\overline\QQ_\ell)
\end{equation}
factors through mixed sheaves
\begin{equation}\label{eq:comp etale vs rigid real4}
\DM(X_\circ,\QQ)\to\Der^b_m(X_\circ,\overline\QQ_\ell)
\end{equation}
simply because \eqref{eq:comp etale vs rigid real3} commutes with the
six operations and $\DM(X_\circ,\QQ)$ is generated by the constant
objects $\QQ(n)$ under the six operations \cite[Cor. 4.4.3]{CD3} and, as proved by
Deligne~\cite{Weil2}, $\Der^b_m(X_\circ,\overline\QQ_\ell)$
is stable under the six operations in $\Der^b_c(X_\circ,\overline\QQ_\ell)$.

The theory of weights of $\overline\QQ_\ell$-sheaves on algebraic
varieties over a finite field provides:
\end{paragr}
\begin{thm}[Beilinson, Bernstein, Deligne, Gabber]\label{thm:BBD}
For any projective morphism $f:Y\to X$, with $Y$ smooth over $k$,
there is an isomorphism
\[
f_*(\overline\QQ_{\ell,Y})
\cong\bigoplus_i{^p\coh^i(f_*(\overline\QQ_{\ell,Y}))}[-i]
\]
and each perverse cohomology object ${^p\coh^i(f_*(\overline\QQ_{\ell,Y}))}$
is semi-simple.
\end{thm}
\begin{proof}
Indeed, we may assume that $f$ is the pullback of a projective morphism
$f_\circ\colon Y_\circ\to X_\circ$ defined over a finite subfield $k_\circ$
of $k$, with $Y_\circ$ smooth over $k_\circ$. We know from \cite[Prop. 6.2.7]{Weil2},
that the functor
\[
(f_\circ)_*:\Der^b_m(Y_\circ,\overline\QQ_\ell)\to\Der^b_m(Y_\circ,\overline\QQ_\ell)
\]
sends $\overline\QQ_{\ell,Y_\circ}$ to a pure object.
By virtue of \cite[Thm. 5.3.8]{BBDG}, pulling back to $X$,
we see that each perverse cohomology object
${^p\coh^i(f_*(\overline\QQ_{\ell,Y}))}$
is semi-simple. Similarly, the decomposition isomorphism
is obtained by applying \cite[Thm.~5.4.5]{BBDG}.
\end{proof}
\begin{cor}
For any $f:Y\to X$ projective with $Y$ smooth over $k$,
the perverse cohomology object ${^p\coh^i(f_*(\overline\QQ_{\ell,Y}))}$
belongs to $\Der^b_\gm(X,\overline\QQ_\ell)$.
\end{cor}
\begin{proof}
The object $f_*(\overline\QQ_{\ell,Y})$ belongs to
$\Der^b_\gm(X,\overline\QQ_\ell)$ and the latter
is idempotent complete by design. 
\end{proof}
\begin{paragr}
We would like to prove that the perverse $t$-structure on
$\Der^b_c(X,\overline\QQ_\ell)$ restricts to a $t$-structure on
$\Der^b_\gm(X,\overline\QQ_\ell)$. The idea is that this should
follow formally from Theorem \ref{thm:BBD} and its corollary. This is the case,
in the sense that we will use abstract methods, but this requires
nevertheless rather sophisticated arguments and thus a little bit of preparation.
\end{paragr}
\begin{thm}[Keller \& Nicol\'a]\label{thm:kellernicola}Let $A$ be a
differential graded $\ZZ$-algebra such that:
\begin{enumerate}
\item[(a)] $\coh^i(A)=0$ for all $i<0$,
\item[(b)] the ring $\coh^0(A)$ is semi-simple, and
\item[(c)] each cohomology group $\coh^i(A)$ is finitely generated as a left
$\coh^0(A)$-module.
\end{enumerate}
Then the stable $\infty$-category $\mathit{Perf}(A)$ admits a bounded
$t$-structure with $\mathit{Perf}(A)^{\leq 0}$ (resp. $\mathit{Perf}(A)^{\geq 0}$)
the smallest full subcategory
containing $A$ closed under extensions, positive (resp. negative) shifts and
direct summands. Furthermore each object of the heart has finite length
and the simple objects of the heart precisely are
the indecomposable direct summands of A in $\mathit{Perf}(A)$.
\end{thm}
\begin{proof}
See \cite[Thm. 8.1]{KeNi}.
\end{proof}
\begin{prop}\label{prop:kellernicola}
Let $R$ be a commutative ring.
Let $\Cat C$ be a $R$-linear stable $\infty$-category
equipped with a bounded $t$-structure.
We assume that the following conditions hold:
\begin{enumerate}
\item[(a)] the heart is  both artinian and noetherian
(i.e. each object is of finite length).
\item[(b)] for all objects $x$ and $y$ of $\Cat C$,
$\Hom_{\Cat C}(x,y))$ is a perfect complex of $R$-modules.
\end{enumerate}
We fix a family $S$ of simple objects of the heart and consider
$\Cat C_S$, the smallest thick stable subcategory of $\Cat C$
containing the elements of $S$. Then the $t$-structure on $\Cat C$
restricts to a $t$-structure on $\Cat C_S$, in such a way that
the inclusion functor $\Cat C_S\to\Cat C$ is $t$-exact.
\end{prop}
\begin{proof}
We prove the case where $S$ consists of finitely many objects.
We can form the object $\sigma=\oplus_{s\in S}s$
and form the algebra of its endomorphisms
\[
A={\Hom(\sigma,\sigma)}\, .
\]
This is a dg $R$-algebra with the following properties:
\begin{enumerate}
\item[(i)] $\coh^0(A)$ is semi-simple.
\item[(ii)] $\coh^i(A)=0$ for $i<0$,
\item[(iii)] it is perfect as a complex of $R$-modules. In particular,
each cohomology group $\coh^i(A)$ is of finite presentation as a $R$-module.
\end{enumerate}
That means that we can apply Theorem \ref{thm:kellernicola}
and get a $t$-structure on $\mathit{Perf}(A)$.
But, by construction, we have $\mathit{Perf}(A)\cong\Cat C_S$,
hence a $t$-structure on $\Cat C_S$,
with $\Cat C_S^{\leq 0}$ (resp. $\Cat C_S^{\geq 0}$)
the smallest full subcategory
containing the elements of $S$,
closed under extensions, positive (resp. negative) shifts and
direct summands. It is thus clear that the inclusion
$\Cat C_S\to\Cat C$ is $t$-exact.

To prove the general case, we write $S$ as a filtered union of finite
subsets $T$, so that we get
\[
\varinjlim_{T\subset S}\Cat C_T\cong\Cat C_S\to \Cat C\, .
\]
Each $\Cat C_T$ has a $t$-structure such that the inclusion
$\Cat C_T\to\Cat C$ it $t$-exact, and each transition map in the diagram
$\{\Cat C_T\}_{T\subset S}$ it $t$-exact. There is thus a unique $t$-structure
on $\Cat C_S$ such that each inclusion $\Cat C_T\to\Cat C_S$ it $t$ exact.
The induced functor $\Cat C_S\to\Cat C$ is then obviously $t$-exact as well.
\end{proof}
\begin{cor}\label{cor:geometric perverse}
For $k$ the algebraic closure of a finite field and $X$ separated of finite
type over $k$, for any prime $\ell\neq\mathit{char}(k)$, the perverse $t$-structure on
$\Der^b_c(X,\overline\QQ_\ell)$ restricts to a $t$-structure on
$\Der^b_\gm(X,\overline\QQ_\ell)$. The heart of
$\Der^b_\gm(X,\overline\QQ_\ell)$ consists precisely of those
perverse sheaves of geometric origin: this is the smallest full
subcategory of the abelian category of perverse sheaves in $\Der^b_c(X,\overline\QQ_\ell)$
closed under finite sums as well as under
subquotients containing objects of the
form ${^p\coh^i(f_*(\overline\QQ_{\ell,Y}))}(n)$ for
any projective morphism $f\colon Y\to X$, with $Y$ smooth over $k$, and $i,n\in\ZZ$.
\end{cor}
\begin{proof}
This is clearly a consequence of Theorem \ref{thm:BBD}
and of Proposition \ref{prop:kellernicola},
with $S$ the collection of those simple perverse sheaves
that occur as a direct factor of ${^p\coh^i(f_*(\overline\QQ_{\ell,Y}))}(n)$
for some projective morphism $f:Y\to X$,
with $Y$ connected and smooth over $k$, and $n\in\ZZ$.
\end{proof}
\begin{rem}
The corollary above means that $\Der^b_\gm(X,\overline\QQ_\ell)$
consists of usual $\overline\QQ_\ell$-sheaves of geometric origin:
those $K$ in $\Der^b_c(X,\overline\QQ_\ell)$ such that each
perverse cohomology object ${^p\coh^i(K)}$ is of geometric origin
as a perverse sheaf.
We can characterize the perverse $t$-structure on $\Der^b_\gm(X,\overline\QQ_\ell)$
as follows (for all $\ell\neq \mathit{char}(k)$).
An object $K$ of $\Der^b_\gm(X,\overline\QQ_\ell)$
belongs to ${^p}\Der^{\leq 0}_\gm(X,\overline\QQ_\ell)$ if and only if there is
a finite stratification of $X$ by locally closed subschemes
such that each stratum $U$ is smooth and connected over $k$, and such that
each restriction of $K$ on $U$ is dualizable (i.e. lisse),
and, for any closed point $x$ of $U$, the fiber of $K$ at $x$
is in $\Der^{\leq -d}(\overline\QQ_\ell)$
where $d$ is the dimension of $U$. In particular, we see that
the induced $t$-structure on $\Der^b_\gm(X,\overline \QQ_\ell)$
is determined by functoriality by the fact that it coincides with the
usual standard $t$-structure on
$\mathit{Perf}(\overline \QQ_\ell)$.
\end{rem}

\begin{rem}\label{rem:Weil etale}
Recall that $\Der^b_\gm(-,\overline{\QQ}_\ell)$, $\Der^b_c(-,\overline{\QQ}_\ell)$
and $\DM_{c}(-,\QQ)$
all are hypersheaves of $\QQ$-linear stable $\infty$-categories
(this follows straight away from \cite[thm.~3.3.37]{CD3}
since the six operations preserve constructibility
in all these cases by \cite[Thm.~4.2.29]{CD3}). Therefore,
many properties of these categories extend to the case where $X$
is any (hyper)sheaf of $\infty$-groupoids on the site of schemes
of finite type over $k$ with respect to the $\h$-topology.

Let us assume that $k$ is the algebraic closure of a finite field
with $q$ elements $k_\circ=\mathbf{F}_q$.
Given any $k_\circ$-scheme of finite type $X_\circ$, we denote by
\[
\phi_X\colon X\to X
\]
the geometric $q$-Frobenius of $X$ (that is a morphism of $k$-schemes).
One defines $X^\mathit{Weil}$ as the quotient of $X$ by the $\ZZ$-action induced by
the $q$-Frobenius $\phi_X$. For any $\h$-hypersheaf of $\infty$-categories $\Der$
over $k$, the $\infty$-category $\Der(X^\mathit{Weil})$ is thus the
$\infty$-category of objects $F$ in $\Der(X)$ equipped with an equivalence
\[
\sigma\colon \phi_X^*(F)\cong F
\]
There is a quotient map $\pi\colon X\to X^\mathit{Weil}_\circ$
inducing a forgetful functor
\[
\pi^*\colon\Der(X^\mathit{Weil}_\circ)\to\Der(X)\ , \quad (F,\sigma)\mapsto F
\]
If moreover $\Der(X)$ is stable and is equipped with a $t$-structure
and if $\phi^*_X$ is $t$-exact, then
there is a unique $t$-structure on $\Der(X^\mathit{Weil}_\circ)$ so that the
forgetful functor above is $t$-exact and induces an equivalence
between the heart of $\Der(X^\mathit{Weil}_\circ)$ and
the abelian category obtained from the above construction
replacing $\Der(-)$ by its heart $\Der(-)^\heartsuit$;
see \cite[Lemma~2.2]{richarz_scholbach2024}.
In particular, we have the stable $\infty$-category
$\Der^b_c(X^\mathit{Weil}_\circ,\overline{\QQ}_\ell)$ of
$\ell$-adic Weil sheaves on $X_\circ$.
It contains the full subcategory
$\Der^b_\gm(X^\mathit{Weil}_\circ,\overline{\QQ}_\ell)$.
The latter consists of those Weil sheaves whose underlying sheaf on $X$ is
constructible of geometric origin. Both $\Der^b_c(X^\mathit{Weil},\overline{\QQ}_\ell)$
and $\Der^b_\gm(X^\mathit{Weil},\overline{\QQ}_\ell)$ have a
perverse $t$-structure
and the inclusion
\[
\Der^b_\gm(X^\mathit{Weil}_\circ,\overline{\QQ}_\ell)\subset
\Der^b_c(X^\mathit{Weil}_\circ,\overline{\QQ}_\ell)
\]
is $t$-exact. This makes precise the slogan that understanding
$\ell$-adic cohomology of smooth schemes of finite type over 
the finite field $k_\circ$ amounts to unerstanding the system
of coefficients 
\[
X_\circ\mapsto\Der^b_\gm(X^\mathit{Weil}_\circ,\overline{\QQ}_\ell)\, .
\]
Therefore, standard conjectures on motives should imply through tannakian arguments
that the coefficient system $\Der^b_\gm(X^\mathit{Weil}_\circ,\overline{\QQ}_\ell)$
itself
should be independent of~$\ell$.
\end{rem}

\section{Counter-examples in mixed characteristic}
The goal of this section is to prove that Proposition~\ref{prop:petit everest}
is not true in mixed characteristic.
\begin{thm}\label{thm:non-everest}
Let $S$ be a noetherian scheme of dimension $1$
and generically of characteristic zero.
We consider a closed point $s$ in the regular locus of $S$.
We assume furthermore that the residue field of $s$ is
algebraically closed and of positive 
characteristic, and that the local ring $\mathcal{O}_{S,s}$ has
at most the cardinality of the continuum. We denote by
\[
i\colon s\to S
\]
the induced immersion and we choose a prime number $\ell$ that is distinct from
the characteristic of the residue field of $s$.
\begin{itemize}
\item[(i)] If $j\colon U\to S$ denotes the complement of $s$ in $S$
then the canonical comparison map
\[
j_!(\ZZ_U)\otimes^\derL\ZZ_{\ell,S}\to j_!(\ZZ_U)_\ell
\]
is not an equivalence.
\item[(ii)] The canonical comparison map
\[
i^*(\ZZ_{\ell,S})\to\ZZ_{\ell,s}
\]
is not an equivalence.
\end{itemize}
\end{thm}
\begin{paragr}
Recall that
the $\infty$-category $\DM(X,\QQ)$ can be described as
$$\DM(X,\QQ)\cong\DA(X,\QQ)$$
where $\DA(X,\QQ)$ is the $\mathbf{P}^1$-stablilization of the
$\mathbf{A}^1$-localization of the derived \mbox{$\infty$-cat}\-egory of
\'etale sheaves of $\QQ$-vector spaces on the smooth-\'etale
site of $X$~\cite[Theorem 16.2.18]{CD3}.
We have
\begin{equation}
\QQ_{\ell,X}=\ZZ_{\ell,X}\otimes\QQ
\end{equation}
which is now an $E_\infty$-ring spectrum in $\DM(X,\QQ)$.
\end{paragr}
\begin{paragr}\label{paragr:rigid}
We will denote by
\[
R=\widehat{\mathcal{O}}_{S,s}
\]
the completion of the local ring ${\mathcal{O}}_{S,s}$ and define
$f\colon T=\spec R\to S$ as the canonical map. Let $K$ be the
field of fractions of $R$ and $\eta=\spec K$.\\

Since the local ring $\mathcal{O}_{S,s}$ has at most the cardinality of the continuum,
the same is true for its completion $R$ as well as for $K$ itself.
Using the axiom of choice, we choose, once and for all, an
embedding of $K$ into the field of complex numbers.
\begin{equation}
K\subset\CC
\end{equation}
We thus get the following diagram of schemes
\begin{equation}\label{basic diagram}
\begin{split}
\xymatrix{
&&\overline{\eta}\ar[d]^{\pi}\\
s\ar[r]^{\imath}\ar[dr]^i&T\ar[d]^f&\eta\ar[l]_{\jmath}\\
&S
}
\end{split}
\end{equation} 
where $\pi:\overline{\eta}=\spec\CC\to\eta$ is the map induced by $K\subset\CC$
and $\imath\colon s\to T$ as well as $\jmath\colon\eta\to T$ are the canonical
embeddings.
\end{paragr}
\begin{paragr}
We define $\mathit{Betti}$ as the $E_\infty$-ring spectrum in $\DM(\overline{\eta},\ZZ)$
representing Betti cohomology with $\ZZ$-coefficients,
obtained by applying the right adjoint of Ayoub's Betti
realization \cite{Betti} to the
constant sheaf, for instance; see \cite[Example 17.1.7]{CD3}.
For any commutative ring $\Lambda$, its $\Lambda$-linear version
simply is
\[
\mathit{Betti}_\Lambda=\mathit{Betti}\otimes^\derL\Lambda\, .
\]
There is also the analytic de Rham spectrum $\mathit{DR}_{an}$
representing analytic de Rham cohomology of smooth complex algebraic
varieties \cite[Example 17.2.22]{CD3}.
There is a canonical comparison isomorphism
\begin{equation}
\mathit{Betti}_\QQ\otimes_\QQ\CC\cong \mathit{DR}_{an}
\end{equation}
because one can always promote morphisms of mixed Weil cohomologies
to isomorphisms of rings spectra \cite[Cor.~17.2.15]{CD3}.

Let $\ell$ be a prime number distinct from the
characteristic of $k$. There is a functorial isomorphism
$$\derR\Gamma(X_\et,\ZZ/\ell^\nu)\cong\derR\Gamma(X(\CC),\ZZ/\ell^\nu\ZZ)$$
inducing a functorial morphism of $E_\infty$-algebras
$$\derR\varprojlim_\nu\derR\Gamma(X_\et,\ZZ/\ell^\nu)\otimes\QQ\cong
\Big(\derR\varprojlim_\nu\derR\Gamma(X(\CC),\ZZ/\ell^\nu\ZZ)\Big)\otimes\QQ
\leftarrow
\derR\Gamma(X(\CC),\ZZ)\otimes\QQ$$
This defines an isomorphism of mixed Weil cohomologies
\begin{equation}
\mathit{Betti}_\QQ\otimes\QQ_\ell\cong\QQ_{\ell,\overline{\eta}}
\end{equation}
by virtue of \cite[Cor.~17.2.15]{CD3}.
\end{paragr}

\begin{paragr}\label{paragr:rigid_DR}
There is algebraic de Rham cohomology $\mathit{DR}$,
the $E_\infty$-ring spectrum in the category $\DM(\eta,\QQ)$ that represents
algebraic de Rham cohomology of smooth algebraic varieties over $K$.
We have a canonical isomorphism of $E_\infty$-ring spectra
\begin{equation}
\mathit{DR}\otimes_K\CC\cong\pi_*\mathit{DR}_{\mathrm{an}}
\end{equation}
which is induced by the obvious maps
$$\derR\Gamma(X,\Omega^*_{X/K})\otimes_K\CC\to
\derR\Gamma(X^{\mathrm{an}},\Omega^*_{X^{\mathrm{an}}/\CC})$$
and applying \cite[Cor.~17.2.15]{CD3}.
There is also the Monsky-Washnitzer cohomology $\mathit{MW}$ in $\DM(T,\QQ)$
that is isomorphic to $\imath_*\mathit{Rig}$, where $\mathit{Rig}$ is the
$E_\infty$-ring spectrum that represents rigid cohomology in $\DM(s,\QQ)$.
The Berthelot-Ogus map
\begin{equation}
\jmath_*\mathit{DR}\to \mathit{MW}\cong \imath_*\mathit{Rig}
\end{equation}
is a morphism of $K$-linear $E_\infty$-ring spectra in $\DM(T,\QQ)$
(see \cite[Theorem 3.2.6]{CD2}).
Therefore,
after choosing a morphism of fields $\QQ_\ell\to\CC$
we obtain morphisms of $E_\infty$-ring spectra
in $\DM(S,\QQ)$:
\begin{equation}\label{equ:compare rig0}
\jmath_*\pi_*\QQ_{\ell,\overline{\eta}}\otimes_{\QQ_\ell}\CC
\cong \jmath_*\pi_*\mathit{Betti}_\QQ\otimes\CC\cong \jmath_*\pi_*\mathit{DR}_{\mathrm{an}}
\cong \jmath_*\mathit{DR}\otimes_K\CC\to \imath_*\mathit{Rig}\otimes_K\CC\, .
\end{equation}
There are also canonical maps of $E_\infty$-ring spectra
\begin{equation}\label{equ:compare rig1}
\QQ_{\ell,T}\to\jmath_*\QQ_{\ell,\eta}\To \jmath_*\pi_*\QQ_{\ell,\overline{\eta}}\, .
\end{equation}
Composing \eqref{equ:compare rig0}
and \eqref{equ:compare rig1}, we thus get
a morphism of $E_\infty$-spectra
\begin{equation}\label{equ:compare rig2}
\QQ_{\ell,T}\To \imath_*\mathit{Rig}\otimes_K\CC\, .
\end{equation}
By adjunction, this defines a morphism of $E_\infty$-ring spectra
\begin{equation}\label{equ:compare rig3}
\imath^*\QQ_{\ell,T}\To\mathit{Rig}\otimes_K\CC
\end{equation}
in $\DM(S,\QQ)$. 
\end{paragr}

\begin{proof}[Proof of Theorem~\ref{thm:non-everest}]
Using the morphism of cofiber sequences
$$\xymatrix{
j_!(\ZZ_{\ell,\eta})\ar[r]\ar[d]&\ZZ_{\ell,S}\ar[r]\ar@{=}[d]&i_*i^*(\ZZ_{\ell,S})\ar[d]\\
j_!(\ZZ_\eta)_\ell\ar[r]&\ZZ_{\ell,S}\ar[r]&i_*(\ZZ_{\ell,s})
}$$
together with the fact that $i_*$ is a fully faithful right adjoint,
we see that assertions (i) and (ii) are equivalent.
We will prove assertion (ii).
Let us assume that we have a canonical isomorphism
of $E_\infty$-ring spectra of the form
\begin{equation}\label{equ:compare rig4}
i^*\QQ_{\ell,S}\cong \QQ_{\ell,s}\, .
\end{equation}
Using the canonical map $f^*\QQ_{\ell,S}\to\QQ_{\ell,T}$,
this produces a morphism of motivic $E_\infty$-ring spectra of the form
\begin{equation}\label{equ:compare rig4bis}
\QQ_{\ell,s}\cong
i^*\QQ_{\ell,S}\cong \imath^*\, f^*\QQ_{\ell,S}\to\imath^*\QQ_{\ell,T}\, .
\end{equation}
Then map \eqref{equ:compare rig3} and morphism \eqref{equ:compare rig4bis}
induce a $\CC$-linear morphism of $E_\infty$-ring spectra
$$\QQ_{\ell,s}\otimes_{\QQ_\ell}\CC\To\mathit{Rig}\otimes_K\CC\, .$$
between mixed Weil cohomologies. By
\cite[Cor.~17.2.15 and Prop.~17.2.16]{CD3}, this is thus an isomorphism
of mixed Weil cohomologies (one can check that the reconstruction theorem provided by \cite[Prop.~17.2.16]{CD3} is functorial).
In particular, for any smooth scheme $X$ over $s$ and any integer $i$, we get a functorial
isomorphism of the form:
\begin{equation}\label{equ:compare rig5}
\coh^{j}_\et(X,\QQ_\ell(i))\otimes_{\QQ_\ell}\CC\cong\coh^j_\rig(X/K)\otimes_K\CC\, .
\end{equation}
For $X$ of good reduction, we may choose a smooth and proper model $Y$ of X over $T$.
We have the following commutative diagram 
\[
\xymatrix{
\coh^j(Y,f^*\QQ_{\ell,S})\ar[r]\ar[d]&\coh^j_\et(Y,\QQ_\ell)\ar[r]\ar[d]&
\coh^j(Y,\imath_*\mathit{Rig})\otimes_K\CC\ar[d]^\cong\\
\coh^j(X,\imath^*f^*\QQ_{\ell,S})\ar[r]\ar[dr]_\cong
&\coh^j_\et(X,\imath^*\QQ_{\ell,T})\ar[r]\ar[d]&
\coh^j_\rig(X/K)\otimes_K\CC\\
&\coh^j_\et(X,\QQ_\ell)\ar[ur]_{\text{\eqref{equ:compare rig5}}}&
}
\]
in which the vertical maps
relating the first line to the second one are induced by applying the functor
$\imath^*\colon\DM(T)\to\DM(s)$ to the diagram
\[
f^*\QQ_{\ell,S}\to\QQ_{\ell,T}\to\imath_*\mathit{Rig}\otimes_K\CC\, .
\]
Moreover, by the proper base change formula, we know that the middle vertical
composed map is an isomorphism:
\[
\coh^j_\et(Y,\QQ_\ell)\cong\coh^j_\et(X,\QQ_\ell)\, .
\]
Using \eqref{equ:compare rig0}, this means that
isomorphism \eqref{equ:compare rig5} is in this case
a composition of classical isomorphisms: specialization relating
$\coh^j_\et(X,\QQ_\ell)$ and $\ell$-adic cohomology of the generic fiber of $Y$,
identification of the latter with Betti cohomology, isomorphism with analytic
de Rham cohomology, Grothendieck's isomorphism between analytic and
algebraic de Rham cohomology, and finally the Berthelot--Ogus isomorphism with
crystalline cohomology of the special fiber.
Note furthermore that isomorphism \eqref{equ:compare rig5} relies on one choice
only: the choice of an embedding of fields $\QQ_\ell\subset\CC$.
Since $\coh^j_\rig(X/K)\otimes_K\CC$ does not depend on the choice of such an embedding,
this means that the action of the
group of automorphisms of the field~$\overline{\QQ}_\ell$
on 
\[
\coh^j_\et(Y\times_T\bar{\eta},\overline{\QQ}_\ell)
\cong\coh^j(Y(\CC),\overline{\QQ}_\ell)\cong
\coh^j(Y(\CC),\QQ)\otimes\overline{\QQ}_\ell
\]
is functorial in the (pure) motive of $X$.
In particular, we can consider the case where
$X$ a supersingular elliptic curve over $s$ that is the fiber
of some elliptic curve $Y$ over $S$.
In this case, the action of $\mathrm{End}(X)$ on
\[
\coh^j_\et(Y,\overline{\QQ}_\ell)
\cong\coh^j_\et(X,\overline{\QQ}_\ell)
\]
induces an action on the
$\QQ$-vector subspace $\coh^j(Y(\CC),\QQ)$. Since this is known not to be true for $j=2$,
we have reached a contradiction.
\end{proof}
\bibliographystyle{amsalpha}
\bibliography{lfin}
\end{document}